\documentclass[12pt]{extarticle}
\usepackage{amsmath, amsthm, amssymb, hyperref, color}
\usepackage[shortlabels]{enumitem}
\usepackage{graphicx}
\usepackage[all]{xypic}
\usepackage{makecell}
\usepackage[final]{pdfpages}
\setboolean{@twoside}{false}
\usepackage{pdfpages}
\usepackage{caption}
\usepackage{subcaption}
\usepackage{scalefnt}
\usepackage{verbatim}
\oddsidemargin \evensidemargin
\newtheorem{theorem}{Theorem}
\numberwithin{theorem}{section}
\newtheorem{proposition}[theorem]{Proposition}

\newtheorem{corollary}[theorem]{Corollary}
\newtheorem{definition}[theorem]{Definition}

\newtheorem{example}[theorem]{Example}
\newtheorem{conjecture}[theorem]{Conjecture}

\usepackage{amsmath}
\usepackage{blkarray} 
\usepackage{amssymb}
\usepackage{amsthm}
\usepackage{mathrsfs}
\usepackage{mathtools}
\usepackage{tikz}
\usepackage{cprotect}
\usepackage{float}

\def\R{{\mathbb R}}
\def\Q{{\mathbb Q}}

\title{Gram Determinants of Real Binary Tensors}

\author{Anna Seigal}

\date{}

\begin{document}

\maketitle

\begin{abstract}

A binary tensor consists of $2^n$ entries arranged into hypercube format $2 \times 2 \times \cdots \times 2$. There are $n$ ways to flatten such a tensor into a matrix of size $2 \times 2^{n-1}$. For each flattening, $M$, we take the determinant of its Gram matrix, ${\rm det }(M M^T)$. We consider the map that sends a tensor to its $n$-tuple of Gram determinants. We propose a semi-algebraic characterization of the image of this map. This offers an answer to a question raised by Hackbusch and Uschmajew concerning the higher-order singular values of tensors.

\end{abstract}

\section{Introduction} \label{1}

The Gram determinants of a real binary tensor of format $2 \times 2 \times \cdots \times 2$ ($n$ times) are an $n$-tuple of quadratic invariants of the tensor. We introduce the {\em Gram locus}, the locus of  tuples that arise as the Gram determinants of a real binary tensor. Here, the Gram locus is equal to the ``set of feasible higher-order singular values'', from \cite{HU}, under change of coordinates. The Gram determinants offer  a convenient set of coordinates for studying the higher order singular values of a tensor.

In Theorem \ref{conv} we find the convex hull of the Gram locus for real binary tensors. It is a convex polytope that we describe explicitly. Its facet defining inequalities are that each Gram determinant is bounded by the sum of the others. We give a sum-of-squares proof. In Theorem \ref{thm:neq3}, we express the Gram locus as a semi-algebraic set for the case of $2 \times 2 \times 2$ tensors. The semi-algebraic description determines whether a tuple lies in the Gram locus or its complement, and characterizes tuples on the boundary. The non-linear part of the boundary of the Gram locus is Figure \ref{fig1}, and it is depicted in the highest higher order singular value coordinates in Figure \ref{fig3}. Examining tensors on the boundary gives a counter-example to a Conjecture stated in Section 1 of \cite{HU}: Example \ref{counter} is a tensor which lies on the boundary of the feasible set, but whose higher order singular values in each flattening are distinct. Its singular values are located at the black dot in Figure \ref{fig3}.

Conjecture \ref{conj:main} proposes the general form for the Gram locus. It has a concise expression as the non-negativity of a single polynomial in the Gram determinants. 

Finally, Section \ref{4} gives a partial answer to \cite[Problem 1.6]{HU}, characterizing the tensors whose higher order singular values coincide. In the case of matrices, agreement of singular values implies orthogonal equivalence. Theorem \ref{2x2x2} shows that the hyperdeterminant bridges the gap between orthogonal equivalence of tensors and the higher order singular value decomposition in the $2 \times 2 \times 2$ case. The $2 \times 2 \times 2$ tensor format is described in the following example.


\begin{example} \label{ex1}The $2 \times 2 \times 2$ tensor $(a_{ijk})$, $0 \leq i,j,k \leq 1$, has eight entries which populate the vertices of the three-cube. It has three flattenings, each of size $2 \times 4$:
$$ \begin{bmatrix} a_{000} & a_{001} & a_{010} & a_{011} \\ a_{100} & a_{101} & a_{110} & a_{111} \end{bmatrix} \qquad \begin{bmatrix} a_{000} & a_{001} & a_{100} & a_{101} \\ a_{010} & a_{011} & a_{110} & a_{111} \end{bmatrix} 
\qquad
\begin{bmatrix} a_{000} & a_{010} & a_{100} & a_{110} \\ a_{001} & a_{011} & a_{101} & a_{111}  \end{bmatrix} .$$
For the $i$th flattening $M$ we find $d_i : = {\rm det}(M M^T)$. For instance, the first Gram determinant is
\begin{small}
$$ d_1 = {(a_{000} a_{101} - a_{001} a_{100})}^2 + {(a_{000} a_{110} - a_{010} a_{100} )}^2 + {(a_{000} a_{111} - a_{011} a_{100} )}^2 + $$
$$ \hspace{4.2ex} {(a_{001} a_{110} - a_{010} a_{101} )}^2 + {(a_{001} a_{111} - a_{011} a_{101} )}^2 + {(a_{010} a_{111} - a_{011} a_{110} )}^2 ,$$
\end{small}
A computation reveals that the linear combination $d_2 + d_3 - d_1$ can be written as a sum of three squared terms:
$$ 2{(a_{000}a_{011} - a_{010}a_{001})}^2 + 2{(a_{100}a_{111} - a_{110}a_{101})}^2 + {( a_{010}a_{101} + a_{001}a_{110} - a_{011}a_{100} - a_{000}a_{111})}^2$$
The sum-of-squares certificate certifies that the expression is non-negative for all real values of the variables.
\end{example}

Take a tensor of format $2 \times 2 \times \cdots \times 2$. Each principal flattening is a matrix with two rows and $2^{n-1}$ columns, obtained by combining the indices from all but one direction. Denoting the rows by vectors ${\bf v}$ and ${\bf w}$, the Gram matrix is
$$ \begin{bmatrix} \leftarrow & {\bf v} & \rightarrow \\ \leftarrow & {\bf w} & \rightarrow \end{bmatrix} \cdot \begin{bmatrix} \uparrow & \uparrow \\ {\bf v} & {\bf w } \\ \downarrow & \downarrow \end{bmatrix} = \begin{bmatrix} {|| {\bf v} ||}^2 & {\langle {\bf v} , {\bf w} \rangle} \\  {\langle {\bf v} , {\bf w} \rangle} & {|| {\bf w} ||}^2 \end{bmatrix} .$$
Its determinant is given by the Cauchy-Schwarz expression $ {||{\bf v}||}^2 {||{\bf w}||}^2 - {\langle {\bf v} , {\bf w} \rangle}^2 $. For the $i$th flattening, this is the $i$th {\em Gram determinant}, denoted $d_i$. By the Cauchy-Binet formula, it is the sum of squares of the $2 \times 2$ minors of the $i$th flattening matrix.  

\begin{definition} Let $n \geq 2$. Consider real binary tensors of format $2 \times 2 \times \cdots \times 2$ ($n$ times). The map $\mathcal{G}$ sends a real binary tensor to its tuple of $n$ Gram determinants:
$$ \mathcal{G} : \R^{2} \otimes \cdots \otimes \R^{2} \to \R^n $$
$$ \hspace{19.5ex} ( a_{ij\ldots k} ) \mapsto ( d_1, \ldots, d_n ) .$$
The map scales by a constant factor under rescaling the input tensor. We define the {\em Gram locus} to be the image $\mathcal{G}(\mathcal{B})$, where $\mathcal{B}$ is the unit ball of tensors whose norm does not exceed one:
$$ \mathcal{B} = \left\{ ( a_{i j \ldots k}) \in \R^{2} \otimes \cdots \otimes \R^{2} : \sum_{i j \ldots k} a_{i j \ldots k}^2 \leq 1 \right\} .$$
\end{definition}

Each Gram determinant $d_i$ is a polynomial of degree four in the entries of the tensor; the map $\mathcal{G}$ is given by $n$ homogeneous degree four polynomials.

\bigskip

The Gram determinant map $\mathcal{G}$ gives the higher order singular values of a binary tensor, as follows. The higher order singular values of a tensor, introduced in \cite{LMV}, are the singular values of its $n$ principal flattenings, the non-negative square roots of the eigenvalues of the $n$ Gram matrices. Just as the singular values of a matrix describe it up to orthogonal change of basis, via the singular value decomposition (SVD), the higher order singular values give the corresponding multilinear structure of a tensor, via the higher order singular value decomposition (HOSVD). The trace of any Gram matrix, $t$, is the sum of the squares of the entries of the original tensor, its squared Frobenius norm, hence is unchanged by the choice of flattening. Thus the higher order singular values from the $i$th flattening are the non-negative solutions to the univariate polynomial in $x$: 
$$ x^4 - t x^2 + d_i .$$
Therefore, the map that sends a binary tensor to its higher order singular values is obtained by composing $\mathcal{G}$ with $n$ maps
\begin{equation}\label{coord} d_i \mapsto \left( \sqrt{\frac{t + \sqrt{t^2 - 4d_i}}{2}} , \sqrt{\frac{t - \sqrt{t^2 - 4d_i}}{2}} \right) .\end{equation}
The first coordinate of this map sends $d_i$ to the largest higher order singular value in that flattening. Characterizing feasible combinations of higher order singular values is an open problem \cite{HU}. In this paper, we use the Gram locus to make first steps towards solving it. 


For any combination of Gram determinants, a dimension count shows there generically exists a $(2^n - n)$-dimensional family of complex tensors whose image under $\mathcal{G}$ is those determinants. We seek a {\em real} tensor in the pre-image. The parts of the image of $\mathcal{G}$ where some $d_i$ almost vanishes are of particular interest: these are tensors which can be approximated to good accuracy by a tensor of smaller flattening rank than its dimension, as in \cite{LMV2}.

We note that the analogue of the Gram determinants can also be studied in the case of hierarchical tensor representations, including the Tensor Train (TT) / Matrix Product State representation \cite[Chapter 12]{H}. The size of the TT format required to represent a tensor is given by the ranks of the flattenings obtained by grouping the first $j$ indices for the rows, with the remaining indices forming the columns, for $1 \leq j \leq d-1$.
A tensor being representable by a TT format with deficient $j$th bond dimension is equivalent to the determinant of that flattening vanishing.

\bigskip

The Gram locus $\mathcal{G}(\mathcal{B})$ is not convex. A natural first outer approximation is its convex hull, which the following theorem describes.

\begin{theorem}\label{conv}
Let $n \geq 2$, and take the map $\mathcal{G}$ and the unit ball $\mathcal{B}$ as above. The boundary of the convex hull of the Gram locus $\mathcal{G}(\mathcal{B})$ is described by the following linear inequalities in the determinants $d_i$:
$$ d_i \leq \sum_{j \neq i} d_j ,\qquad 0 \leq d_i \leq \frac{1}{4} , \qquad 1 \leq i \leq n .$$
In particular, this is a convex polytope with $2^n - n$ vertices, namely the point $(0,\ldots,0)$ and all points $(\frac{1}{4},\ldots,\frac{1}{4},0,\ldots,0)$ consisting of any $i \geq 2$ coordinates $\frac{1}{4}$, and the remaining coordinates zero.
\end{theorem}

When $n = 2$, we are in the case of a $2 \times 2$ matrix. It is well known that the two Gram determinants are equal. The inequalities simplify to $0 \leq d_1 = d_2 \leq \frac{1}{4}$.

The constant bounds on the Gram determinants constrain them to the cube $[0,\frac{1}{4}]^n$. The other linear inequalities are satisfied on a proportion $1 - \frac{1}{(n-1)!}$ of this cube. The volume occupied by tuples with $d_1 \geq d_2 + \cdots + d_n$ is $\frac{1}{n!}$, and there are $n$ such regions that are excluded overall, a total volume of $\frac{1}{(n-1)!}$. This fraction is also the proportion of an arbitrarily small neighborhood of zero that is satisfied by the linear inequalities. 

The true image $\mathcal{G}(\mathcal{B})$ is a semi-algebraic subset of the convex hull. Its description relies on two polynomials. The first polynomial is the product of the linear conditions above:
$$ Q_1 = \prod_{i = 1}^n \left( \sum_{j \neq i} d_j  - d_i \right).$$
Inside the positive orthant, the non-negativity of $Q_1$ is equivalent to the non-negativity of each of its linear factors. The second polynomial is given by the following product of linear factors in the $\sqrt{d_i}$:
$$ Q_2 = \frac{1}{2} \times \prod_{ {i, j, \ldots ,k} \in \{ \pm 1\} } (i \sqrt{d_1} + j \sqrt{d_2} + \cdots + k \sqrt{d_n} ) .$$
This is a product of $2^n$ terms, yielding a polynomial of degree $2^{n-1}$ in the $d_i$. Each term appears twice in the product, up to global sign change. Hence $Q_2$ is a perfect square.

\begin{theorem}\label{thm:neq3}
Let $n=3$. The Gram locus $\mathcal{G} (\mathcal{B})$ is described, inside the cube $0 \leq d_i \leq \frac{1}{4}$, by the union of the following two semi-algebraic sets:
\begin{enumerate}
\item The region $Q_1 \geq Q_2$
\item The region $Q_1 \leq Q_2$ and
$ {(d_i - d_j)}^2 +\frac{1}{2} ( d_i+ d_j) \leq \frac{3}{16} $
for all $\{ i, j \} \subset \{ 1, 2, 3 \}$. 
\end{enumerate}
\end{theorem}

\begin{conjecture}\label{conj:main}
Let $n \geq 4$. The Gram locus $\mathcal{G} (\mathcal{B})$ is given by $ Q_1 \geq Q_2$ and $0 \leq d_i \leq \frac{1}{4} $ for $i = 1, \ldots, n$.
\end{conjecture}

Section \ref{3} explains why the second algebraic set from Theorem \ref{thm:neq3} doesn't appear in Conjecture \ref{conj:main}.
 Theorem \ref{conv} can be re-stated as saying that the convex hull of $\mathcal{G}(\mathcal{B})$ is given, inside the cube ${[0, \frac{1}{4}]}^n$, by $Q_1 \geq 0$. Since $Q_2$ is a square, the region $Q_1 \geq Q_2$ is strictly contained inside this convex set.

\section{The Convex Hull of the Gram Locus}

We begin by working through the main part of the proof of Theorem \ref{conv} in the case $n=3$.

\begin{example}[$2 \times 2 \times 2$ tensors]\label{neq3}
Consider the entries of the tensor as the vertices of the three-cube:
\vspace{-2ex}
\begin{figure}[H]
\centering
\includegraphics[width=4cm]{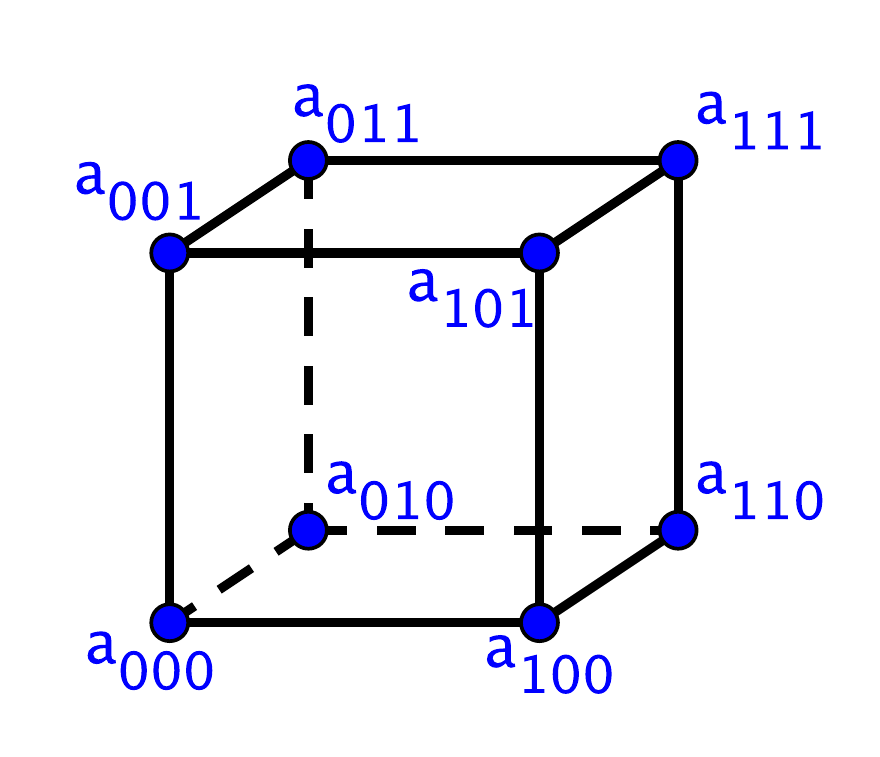}
\end{figure}
\vspace{-2ex}

Each of the three determinants is the sum of squares of the ${{ 4 \choose 2 }} = 6$ minors from a flattening in Example \ref{ex1}. These are combinations of four vertices in the cube. For example, the front face corresponds to the squared minor ${(a_{000} a_{101} - a_{001} a_{100})}^2$. This minor features in flattenings one and three, but not two. Every face of the cube appears as a squared term in two out of three determinants. The faces account for 12 of the 18 minors. The remaining six are unique to one flattening. They are the shaded crimson squares in Figure~\ref{3f}.
\begin{figure}[H]
\centering
\includegraphics[width=5cm]{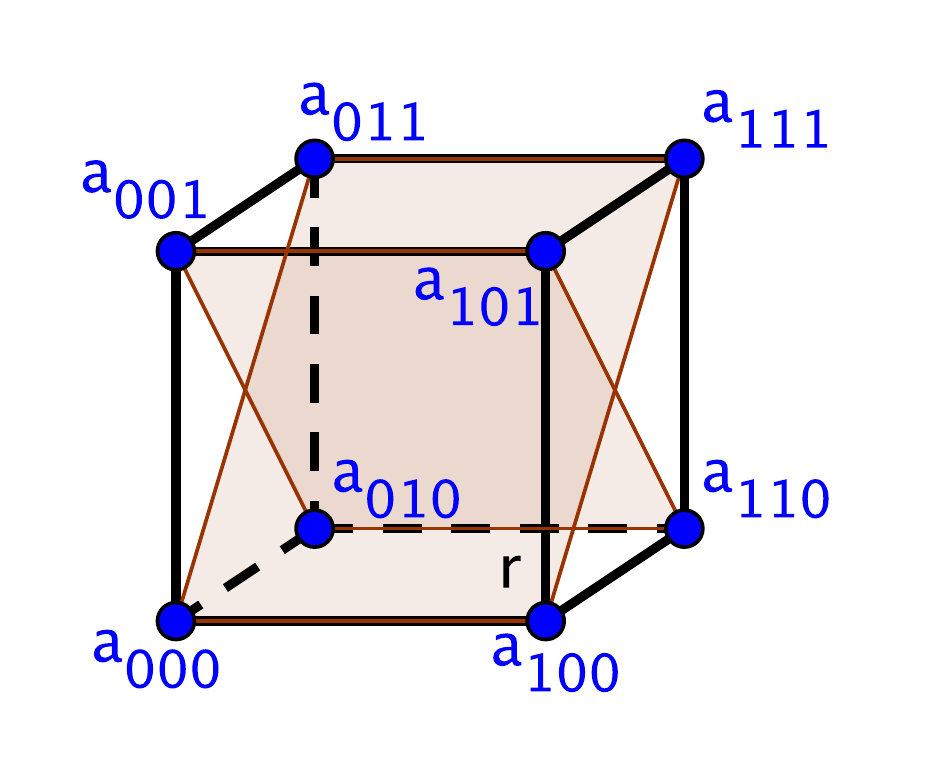}
\includegraphics[width=5cm]{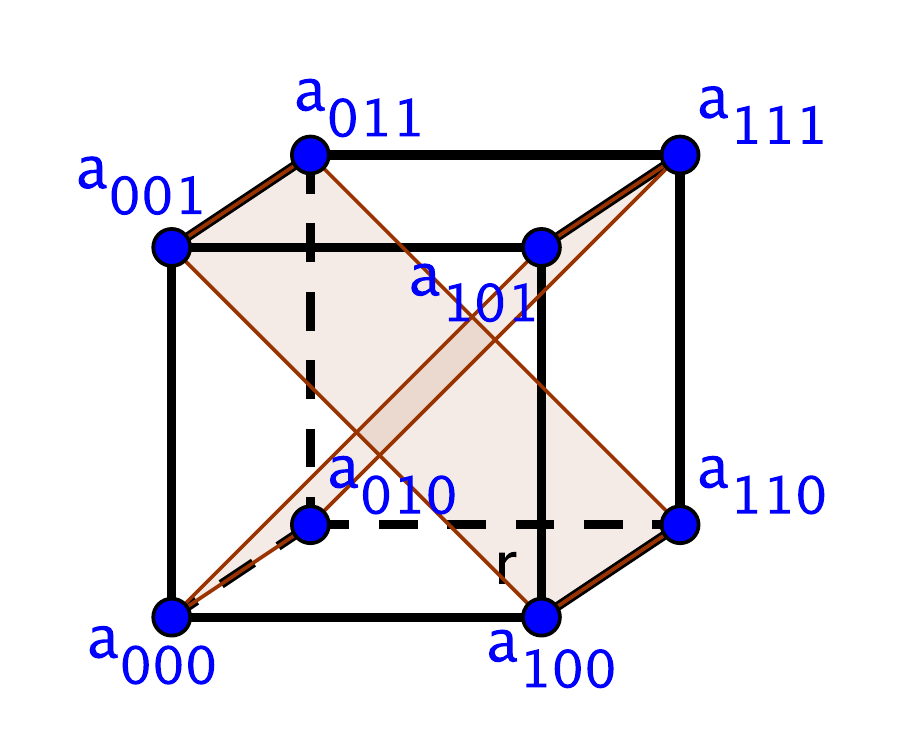}
\hspace{0.7ex}
\includegraphics[width=4.5cm]{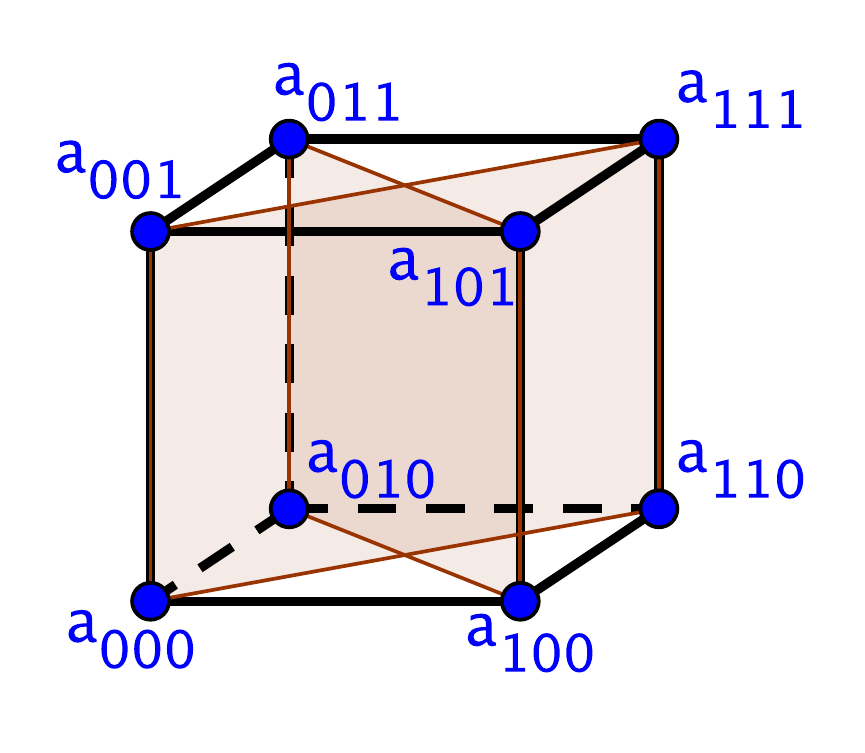}
\caption{The minors unique to flattenings one, two and three respectively}
\label{3f}
\end{figure}

The six crimson minors have monomials $a_{\bf i} a_{\bf j}$ where ${\bf i}$ and ${\bf j}$ are multi-indices in $\{0,1\}^3$ that differ in all three coordinates. Each such term is determined by the two other indices in the term whose first index is zero, so it is represented by a vertex in Figure \ref{square1}. The edges connect monomials that appear in the same minor. \vspace{-2ex}
\begin{figure}[H]
\centering
\includegraphics[width=3cm]{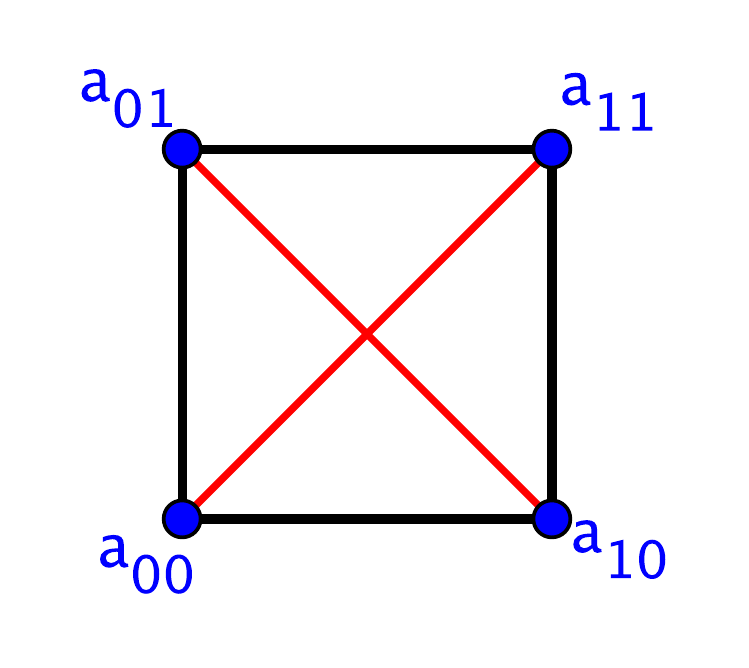}
\vspace{-2ex}
\caption{The minors unique to one flattening are represented by edges. The black edges are minors from flattenings two or three. The red diagonal edges are from flattening one.}
\label{square1}
\end{figure}
\vspace{-1ex}

The three Gram determinants are denoted $d_1$, $d_2$ and $d_3$. As in the statement of Theorem \ref{conv}, we aim to show that
$$ D^{(3)} : = d_2 + d_3 - d_1 \geq 0.$$
The other inequalities follow from $D^{(3)}$ by relabeling. Example \ref{ex1} gives a sum-of-squares certificate for the non-negativity of $D^{(3)}$. Below we carry out the sum-of-squares computation using the notation of the proof of Theorem \ref{conv}. Each determinant is already given by a sum-of-squares expression, and we show how to absorb the subtraction of $d_1$ into the expressions for $d_2$ and $d_3$.

The minors come in two types: the faces of the cube have monomials $a_{\bf i} a_{\bf j}$ where ${\bf i}$ and ${\bf j}$ differ in two indices, and the crimson minors have ${\bf i}$ and ${\bf j}$ differing in all three indices. We write
$$ D^{(3)} = D^{(3)}_2 + D^{(3)}_3 ,$$
where $D^{(3)}_m$ are the minors whose monomials differ in $m$ indices. We find a sum-of-squares certificate for the two pieces $D^{(3)}_2$ and $D^{(3)}_3$ individually.

All terms in $D^{(3)}_2$ that appear in $d_1$ also appear in either $d_2$ or $d_3$, and hence they cancel out in $D^{(3)}_2$. Therefore $D^{(3)}_2$ is a sum-of-squares polynomial consisting of all squared minors in $d_2$ or $d_3$ but not in $d_1$:
$$D^{(3)}_2 = 2{(a_{000}a_{011} - a_{010}a_{001})}^2 + 2{(a_{100}a_{111} - a_{110}a_{101})}^2 .$$

A direct computation shows that the combination of minors depicted in Figure \ref{square1} can be expressed as a perfect square:
$$D^{(3)}_3 = {( a_{010}a_{101} + a_{001}a_{110} - a_{011}a_{100} - a_{000}a_{111})}^2 .$$
These are summed to give a sum-of-squares expression for $D^{(3)}$.
\end{example}

\begin{example}[$2 \times 2 \times 2 \times 2$ tensors]Define
$ D^{(4)}: = d_2 + d_3 + d_4 - d_1 $. We have
$$D^{(4)} = D^{(4)}_2 + D^{(4)}_3 + D^{(4)}_4 $$
where, as above, $D^{(4)}_m$ consists of those minors from $D^{(4)}$ whose monomials $a_{\bf i} a_{\bf j}$ have ${\bf i}$ and ${\bf j}$ differing in $m$ indices. $D^{(4)}_2$ is already in sum-of-squares form:
\begin{footnotesize}
$$2{(a_{0000}a_{0011} - a_{0001}a_{0010} )}^2 + 2{(a_{0000}a_{0101} - a_{0100}a_{0001})}^2 + 2{(a_{0000}a_{0110} - a_{0010}a_{0100})}^2 + $$ $$ 2{(a_{1000}a_{1011} - a_{1001}a_{1010})}^2 + 2{(a_{1000}a_{1101} - a_{1100}a_{1001})}^2 + 2{(a_{1000}a_{1110} - a_{1010}a_{1100} )}^2 + $$
$$ 2{(a_{0100}a_{0111} - a_{0101} a_{0110} )}^2 + 2{(a_{0010} a_{0111} - a_{0110} a_{0011})}^2 + 2{(a_{0001} a_{0111} - a_{0011} a_{0101} )}^2 + $$
$$ 2{(a_{1100} a_{1111} - a_{1101} a_{1110} )}^2 + 2{(a_{1010} a_{1111} - a_{1110} a_{1011} )}^2 + 2{(a_{1001} a_{1111} - a_{1011} a_{1101} )}^2.$$
\end{footnotesize}
The piece $D^{(4)}_3$ has sum-of-squares certificate
\vspace{-1ex} 
\begin{footnotesize}
$${( a_{0100}a_{1010} + a_{0010}a_{1100} - a_{0110}a_{1000} - a_{0000}a_{1110})}^2 + {( a_{0101}a_{1011} + a_{0011}a_{1101} - a_{0111}a_{1001} - a_{0001}a_{1111})}^2 + $$
$$ {( a_{0010}a_{1001} + a_{0001}a_{1010} - a_{0011}a_{1000} - a_{0000}a_{1011})}^2 +  {( a_{0110}a_{1101} + a_{0101}a_{1110} - a_{0111}a_{1100} - a_{0100}a_{1111})}^2  + $$
$$ {( a_{0100}a_{1001} + a_{0001}a_{1100} - a_{0101}a_{1000} - a_{0000}a_{1101})}^2 + {( a_{0110}a_{1011} + a_{0011}a_{1110} - a_{0111}a_{1010} - a_{0010}a_{1111})}^2 +$$
$$ {(a_{0000}a_{0111} - a_{0001}a_{0110})}^2 + {(a_{0000}a_{0111} - a_{0010}a_{0101})}^2 + {(a_{0000}a_{0111} - a_{0100} a_{0011})}^2 + $$
$$ {(a_{1000}a_{1111} - a_{1001} a_{1110} )}^2 + {(a_{1000} a_{1111} - a_{1010} a_{1101} )}^2 + {(a_{1000} a_{1111} - a_{1100} a_{1011} )}^2 + $$
$$ {(a_{0001} a_{0110} - a_{0011} a_{0100} )}^2 + {(a_{0001} a_{0110} - a_{0101} a_{0010} )}^2 + {(a_{0010} a_{0101} - a_{0100} a_{0011} )}^2 + $$
$$ {(a_{1001} a_{1110} - a_{1011} a_{1100} )}^2 + {(a_{1001} a_{1110} - a_{1101} a_{1010} )}^2 + {(a_{1010} a_{1101} - a_{1100} a_{1011} )}^2, $$
\end{footnotesize}
and the final piece $D^{(4)}_4$ has sum-of-squares certificate
\begin{footnotesize}
$${(a_{0010} a_{1101} + a_{0111} a_{1000} - a_{0011} a_{1100} - a_{0110} a_{1001} )}^2 + {( a_{0000} a_{1111} - a_{0001} a_{1110} - a_{0100} a_{1011} + a_{0101} a_{1010} )}^2 + $$
$$ {(a_{0000} a_{1111} + a_{0111} a_{1000} - a_{0010} a_{1101} - a_{0101} a_{1010} )}^2 + {(a_{0100} a_{1011} + a_{0011} a_{1100} - a_{0001} a_{1110} - a_{0110} a_{1001} )}^2,$$
\end{footnotesize}
which we obtain as follows. The monomials in $D^{(4)}_4$ are of the form $a_{\bf i} a_{\bf j}$ where ${\bf i}$ and ${\bf j}$ differ in all four indices. As in Example \ref{neq3}, we can label the vertices of a three dimensional cube by such terms, by writing the indices that occur after the zero in the term that starts with a zero. The minors coming from $d_1$ are the red diagonal edges, with other minors labeled by black edges. We obtain:
\begin{figure}[H]
\centering
\includegraphics[width=4cm]{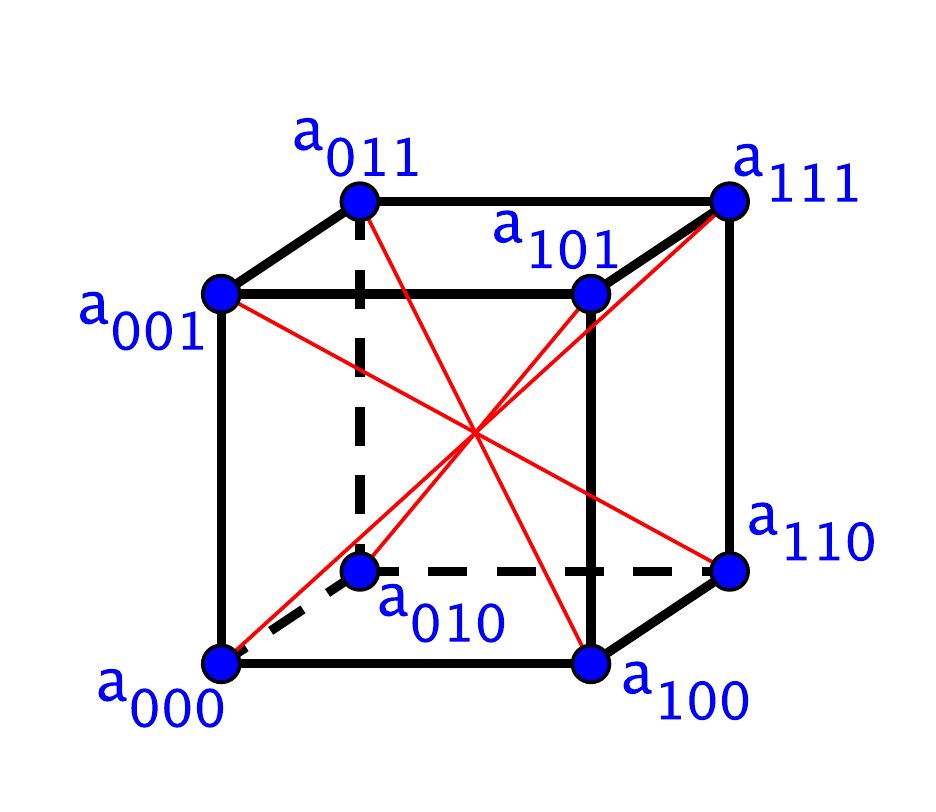}
\end{figure}
To show that the polynomial represented by this picture has a sum-of-squares certificate, we write it as the sum of four pieces whose shape is that in Figure \ref{square1}. Such pieces are the same as $D^{(3)}_3$ up to relabeling, hence they are perfect squares.
\begin{figure}[H]
\centering
\includegraphics[width=4cm]{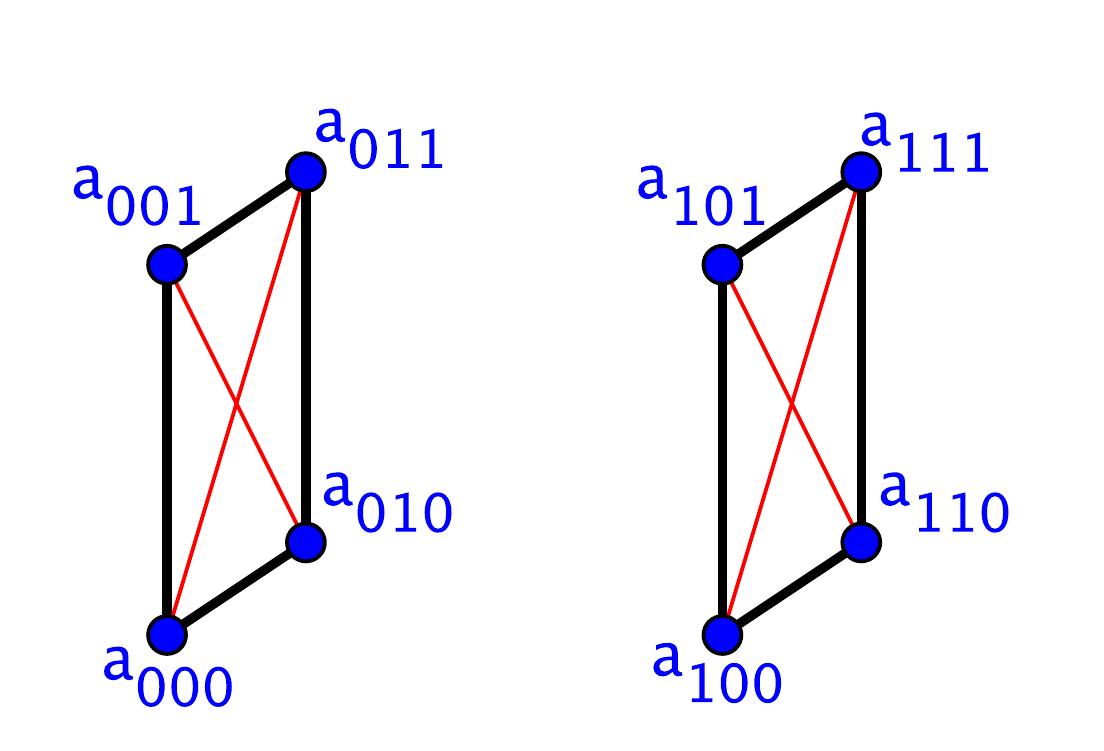} \quad
\includegraphics[width=5.2cm]{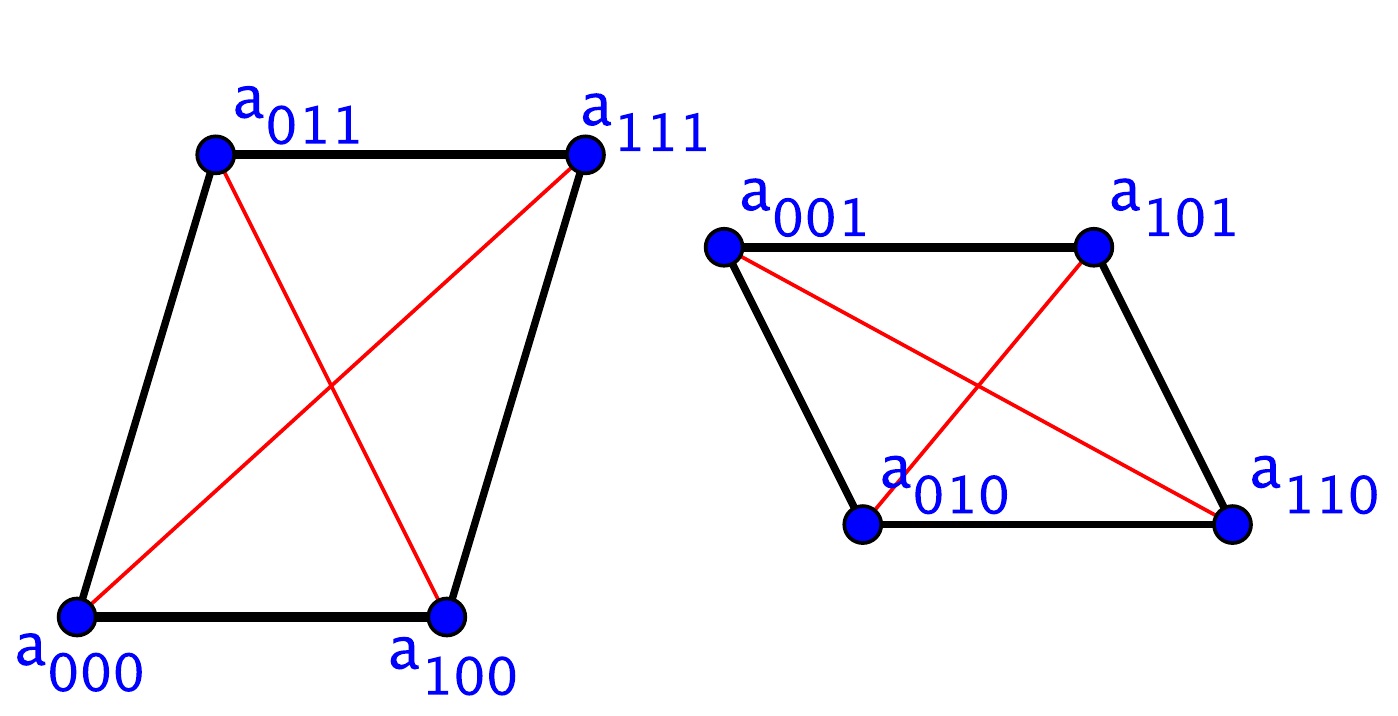}
\end{figure}
\end{example}

\bigskip

We now give the proof of Theorem \ref{conv}, which builds on the above cases via induction.

\begin{proof}[Proof of Theorem \ref{conv}]
With a flattening denoted
$$ \begin{bmatrix} \leftarrow & {\bf v} & \rightarrow \\ \leftarrow & {\bf w} & \rightarrow \end{bmatrix},$$
the trace of the Gram matrix is given by the expression ${||{\bf v}||}^2 + {||{\bf w}||}^2$ and the determinant is $ {||{\bf v}||}^2 {||{\bf w}||}^2 - {\langle {\bf v} , {\bf w} \rangle}^2 $. The Cauchy-Schwarz inequality shows that the lower bound for the determinant is 0. The upper bound is $\frac{1}{4}$, since this is the maximum value taken by the product of two numbers that sum to one. Thus the image is contained in the cube $[0,\frac{1}{4}]^n$.

The vertices of the polytope described by the linear inequalities are: the point $(0,0,\ldots,0) \in \R^n$, and all points consisting of $i$ coordinates $\frac{1}{4}$ and $n-i$ coordinates $0$, with $i \geq 2$. For fixed $i$, there are ${{ n \choose i}}$ such points. We first show that each of these vertices lies in the image $\mathcal{G}(\mathcal{B})$. The determinant tuple $(0,0,\ldots,0) \in \mathcal{S}$ is obtained from any rank one tensor. Consider the tensor $A$ with entries given by
$$ a_{00\ldots0} = \frac{1}{\sqrt{2}}, \quad a_{110\ldots0} = \frac{1}{\sqrt{2}} , \quad a_{ij\ldots k} = 0 \text{, otherwise} .$$
The first two flattenings have one non-zero entry in each of ${\bf v}$ and ${\bf w}$, with the two vectors ${\bf v}$ and ${\bf w}$ orthogonal. Hence the determinants of the corresponding Gram matrices both evaluate to $\frac{1}{4}$. For all other flattenings, ${\bf w} = 0$ and the Gram determinant is zero. Permuting indices, we see that all points with two coordinates $\frac{1}{4}$, and all others equal to zero, are in the image. Modifying the above example, so that the second non-vanishing entry is at $a_{1,1,\ldots,1,0,0,\ldots,0}$, with $i$ indices equal to 1, shows similarly that vertices with $i$ coordinates at $\frac{1}{4}$ are in the image $\mathcal{G}(\mathcal{B})$. This implies that the true convex hull of the Gram locus contains the one in the statement of the theorem.

\bigskip

It remains to show that all other points are outside the image of the map. This is equivalent to showing that the Gram determinants $d_i$ of a real binary tensor satisfy the inequality
\begin{equation} D^{(n)} := d_2 + d_3 + \cdots + d_n - d_1 \geq 0 .\end{equation} \label{lineard}
Each Gram determinant $d_i$ of a tensor $A$ is the sum of squares of the $2 \times 2$ minors of the $i$th flattening of $A$. That is, each $d_i$ is given by a sum-of-squares expression. The polynomial $D^{(n)}$ is degree four in the entries of the original tensor, and we seek a sum-of-squares certificate for it. The set-up is symmetric in the different $d_i$, so this certificate can be re-labeled to give the other parts of the boundary.

We first split up the polynomial $D^{(n)}$ into manageable pieces, and find a sum-of-squares certificate for each piece. The first Gram determinant can be written
$$ d_1 = \sum {( a_{0 {\bf i}} a_{1 {\bf j}} - a_{1 {\bf i}} a_{ 0 {\bf j}} )}^2 $$
where the sum is taken over all ${\bf i}, {\bf j} \subset {\{ 0, 1\}}^{n-1}$ with ${\bf i} \neq {\bf j}$. Similarly, the $k$th determinant is expressible in this form, where instead it is the $k$th index that is swapped in each term. The polynomial $D^{(n)}$ can thus be written in terms of degree two monomials $a_{{\bf i}} a_{\bf j}$, where ${\bf i}, {\bf j} \in {\{ 0, 1\}}^n$, and the multi-indices ${\bf i}$ and ${\bf j}$ differ in at least $2$ locations. For a monomial $a_{\bf i} a_{\bf j}$, let $m$ count the number of locations where ${\bf i}$ and ${\bf j}$ differ (so $2 \leq m \leq n$). Our manageable pieces arise from fixing the value of $m$. The value of $m$ is fixed on each summand of $D^{(m)}$, and we let $D^{(n)}_m$ denote the terms of $D^{(n)}$ with given value of $m$. We seek a sum-of-squares certificate for each piece $D^{(n)}_m$. Their sum gives a certificate for $D^{(n)}$. 

The rest of the proof proceeds as follows. We first find the sum-of-squares certificate for the piece $D^{(n)}_2$. We also examine $D^{(3)}$. Then we show that the polynomial $D^{(n)}_m$, with $m < n$, is equal to a sum of polynomials, each equal to $D^{(m)}_m$ up to relabeling indices. Finally we relate the structure of $D^{(m)}_m$ to $D^{(m-1)}_{m-1}$, and hence can conclude the proof by induction.

Terms in $D^{(n)}_2$ that come from $d_1$ are of the form
$ {( a_{0 {\bf i}} a_{1 {\bf j}} - a_{1 {\bf i}} a_{ 0 {\bf j}} )}^2 $,
where ${\bf i}$ and ${\bf j}$ differ in exactly one location. Without loss of generality, we can assume they differ in their first location, and that ${\bf i} = (0, \ldots )$. We can therefore re-write the above term as
$$ {( a_{0 0 {\bf k}} a_{1 1 {\bf k}} - a_{1 0 {\bf k}} a_{ 0 1 {\bf k}} )}^2, \qquad {\bf k} \in {\{ 0, 1\}}^{n-2} .$$
We observe that this term also appears in $d_2$. Relabeling the above example, we see that all $D^{(n)}_2$-terms in $d_1$ also appear in some other $d_k$, and hence they do not appear in $D^{(n)}_2$. Therefore $D^{(n)}_2$ is a sum-of-squares polynomial: it consists of all squared minors that appear in some $d_k$, $2 \leq k \leq n$, but not in $d_1$.

Now we examine the structure of $D^{(3)}_3$. As in Example \ref{neq3}, a direct computation shows
$ D^{(3)}_3 = { ( a_{011} a_{100} - a_{010} a_{101} - a_{011} a_{110} + a_{000} a_{111} )}^2.$
Combining this with the above sum-of-squares expression for $D^{(n)}_2$, we get a sum-of-squares certificate for $D^{(3)} = D^{(3)}_2 + D^{(3)}_3$.

Next we relate $D^{(n)}_m$ to the polynomial $D^{(m)}_m$ up to relabeling. Consider some term in $D^{(n)}_m$ coming from $d_1$. It is of the form
$$ {( a_{0 {\bf i}} a_{1 {\bf j}} - a_{1 {\bf i}} a_{ 0 {\bf j}} )}^2 , \qquad {\bf i}, {\bf j} \in {\{ 0, 1\}}^{n-1} $$
where ${\bf i}$ and ${\bf j}$ differ in exactly $m-1$ locations. Without loss of generality, we can assume that ${\bf i}$ and ${\bf j}$ differ in their first $m-1$ locations. Forgetting the remaining $n-m$ indices gives a projection onto $D^{(m)}_m$. Repeating for all subsets of $m$ indices gives ${{ n \choose m}}$ copies of $D^{(m)}_m$. We can obtain a sum-of-squares certificate for $D^{(n)}_m$ from one for $D^{(m)}_m$ by re-labeling ${{n \choose m}}$ times and summing.

The rest of the proof is by induction, with the base case $D^{(3)}_3$. For the induction step, we relate $D^{(m)}_m$, where $m \geq 3$, to $D^{(m-1)}_{m-1}$ and $D^{(3)}_3$. We saw above that the polynomial $D^{(m)}$ consists of monomials $a_{\bf i} a_{\bf j}$ with multi-indices ${\bf i}, {\bf j} \in {\{ 0, 1 \}}^m$. Those in $D^{(m)}_m$ have ${\bf i}$ different from ${\bf j}$ in all $m$ locations. For example, the monomial $a_{00\ldots0}a_{11\ldots1}$ appears in $D^{(m)}_m$. For such monomials, the second variable is uniquely determined by the first.

We label the monomials in $D^{(m)}_m$ by ${\{0, 1\}}^{m-1}$ according to the $m-1$ indices that appear after the $0$ in the term that starts with a $0$. These are the $2^{n-1}$ vertices of the following graph. We build an edge between two vertices labeled by ${\bf i}$ and ${\bf j}$ if $a_{0 {\bf i}}$ and $a_{0 {\bf j}}$ appear in the same term in some $d_k$, $1 \leq k \leq n$. Thus each edge of the graph is a summand in $D^{(m)}_m$. The edges are weighted by the coefficient with which the term appears in $D^{(m)}_m$. Those coming from $d_1$ have weight $-1$, while all others have weight $+1$. The positively-weighted edges make the $(m-1)$-dimensional cube. The negatively-weighted edges are the diagonals of this cube. For example, the summand $ {( a_{0 0 0 \ldots 0 } a_{1 1 1 \ldots 1 } - a_{ 1 0 0 \ldots 0} a_{0 1 1 \ldots 1} )}^2 $ contains both $a_{0 0 0 \ldots 0 }$ and $a_{0 1 1 \ldots 1 }$, hence corresponds to the edge between $(0,0,\ldots,0)$ and $(1,1,\ldots,1)$. 

There are $2^{d-2}$ diagonals in the $(m-1)$-dimensional cube. We group them into $2^{d-3}$ pairs, where the two diagonals in a pair differ in their first index. We extract $2^{d-3}$ sub-graphs by considering the edges contained in the four vertices of the two diagonals. We build part of the sum-of-squares certificate from each of the sub-graphs, and then a certificate from the remaining edges. Each sub-graph looks like Figure \ref{sg}.

\begin{figure}[h]
\centering
\includegraphics[width = 2.8cm]{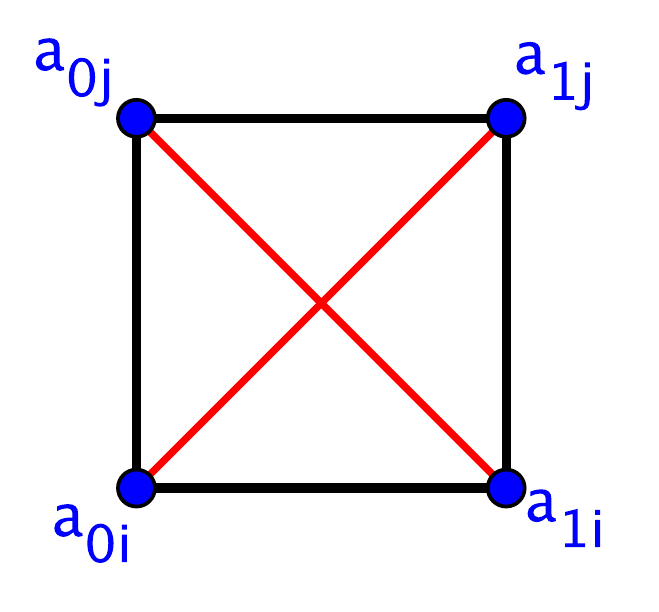}
\caption{A copy of $D_3^{(3)}$ inside $D_m^{(m)}$}
\label{sg}
\end{figure}

The vertical edges in Figure \ref{sg} are positively-weighted in the original graph. The red edges are negatively-weighted in the original graph. The horizontal edges were not in the original graph, but we include them in each sub-graph, at the expense of including them with negative weight among the remaining edges (this ensures they are present with overall weight 0). This graph is the 2-cube with negatively-weighted diagonals. Hence it encodes $D^{(3)}_3$, and thus the polynomial obtained from these sub-graphs has a sum-of-squares certificate.

It remains to consider the structure of the remaining positively and negatively-weighted edges. We have disconnected vertices according to the value of their first index. So we have two cubes of dimension $m-2$. The new negatively-weighted edges are the diagonals of these two smaller cubes. Hence we have two copies of $D^{(m-1)}_{m-1}$. By our induction hypothesis, these both have a sum-of-squares certificate. This concludes the proof. \end{proof}

\begin{proposition}
There are
$$  2^{2n-5} ( 3n - 5) - 2^{n-3} (n^2 - n - 1) $$
terms in the sum-of-squares certificate for $D^{(n)} = d_2 + \cdots + d_n - d_1$ from Theorem \ref{conv}. This formula is valid for all $n \geq 2$.
\end{proposition}

\begin{proof}
Recall that we split up the expression $D^{(n)}$ as the sum $D^{(n)} = \sum_{m = 2}^n D^{(n)}_m$, where $D^{(n)}_m$ consists of those terms containing products $a_{\bf i} a_{\bf j}$ in which the multi-indices ${\bf i}$ and ${\bf j}$ differing in $m$ indices. We count the terms that arise in the sum-of-squares certificate for each $m$.

We first count the terms in the sum-of-squares certificate for $D^{(m)}_m$. From the proof of Theorem \ref{conv}, recall that the sum-of-squares certificate for $D^{(m)}_m$ is made from $2^{d-3}$ copies of the certificate for $D^{(3)}_3$, which is comprised of a single squared term, and two copies of $D^{(m-1)}_{m-1}$. This gives a recursive relationship, whose solution is 
$2^{m-3} (m-2)$.

We now split each $D^{(n)}_m$ up into two  pieces, according to whether the multi-indices ${\bf i}$ and ${\bf j}$ differ in their first index, noting that this property is constant on each term of the certificate. First consider those terms that arise from $D^{(n)}_m$ where ${\bf i}$ and ${\bf j}$ differ in their first index. There are ${{ n-1 \choose m-1}}$ choices for the remaining indices that differ. We fix all other $n - m$ indices at value $0$ or $1$, with $2^{n - m}$ choices. We are then left with a copy of $D^{(m)}_m$. Hence, there are
\begin{equation}\label{m1} {{ n-1 \choose m-1}} 2^{n - m} \cdot  2^{m-3} (m -2) = {{ n-1 \choose m-1}} 2^{n - 3} (m -2) \end{equation}
terms overall.

Next, we consider the terms that arise from variables $a_{\bf i} a_{\bf j}$ where ${\bf i}$ and ${\bf j}$ differ in $m$ locations, and they do not differ in their first index. These are terms contributed solely by $d_2 + \cdots + d_n$, hence they are of the form ${( a_{\bf i} a_{\bf j} - a_{\bf k} a_{\bf l})}^2$ where ${\bf k}$ and ${\bf l}$ are obtained from ${\bf i}$ and ${\bf j}$ by swapping a single index. To count such terms, we first count the number of such pairs $a_{\bf i} a_{\bf j}$ that appear. There are ${{ n-1 \choose m }}$ choices for the $m$ indices at which ${\bf i}$ and ${\bf j}$ differ. Let ${\bf i}$ and ${\bf j}$ be ordered so that ${\bf i}$ is 0 at the first location where they differ. There are $2^{n-1}$ choices for the indices of ${\bf i}$, and these determine those of ${\bf j}$. Furthermore, each term $a_{\bf i} a_{\bf j}$ appears in $m$ times in the certificate. The terms in which it appears are all distinct when $m \geq 3$. Two such pairs comprise each term of the certificate, hence there are
\begin{equation}\label{m2} {{ n-1 \choose m}} 2^{n-2} m \end{equation}
terms when $m \geq 3$. The case $m = 2$ is similar, except that and each term occurs with coefficient two, so we have a count of $ {{ n - 1 \choose 2}} 2^{n-2}$ terms. Summing \eqref{m1} and \eqref{m2} from $m = 3$ to $n$, and including the case $m=2$, we get the desired formula for $n \geq 3$.

The result also holds when $n=2$, but with a different argument. The formula evaluates to $0$ when $n = 2$. The set-up in this case is of a $2 \times 2$ matrix. The two Gram determinants $d_1$ and $d_2$ arise as the determinant of a matrix and its transpose respectively. Hence $d_1 - d_2 = 0$, and $0$ terms suffice for the trivial sum-of-squares certificate. The formula also evaluates to $0$ in the case $n = 1$. \end{proof}

We now consider tensors that map to the boundary of the convex hull of the Gram locus.

\begin{corollary}
Real binary tensors with Gram determinants satisfying
$D^{(n)} = d_2 + \cdots + d_n - d_1= 0$
have only two determinants non-zero: $d_1$ and one other. They are given by the tensor product of a $2 \times 2$ matrix, $M$, with $n-2$ vectors, $v^{(j)}$, according to the formula:
$$ a_{i_1 \ldots i_n} = M_{i_1 i_j} v^{(2)}_{i_2} \cdots \widehat{v^{(j)}_{i_j}}\cdots v^{(n)}_{i_n} ,$$
where $\widehat{v^{(j)}_{i_j}}$ denotes the omission of the $j$th term from the product. Conversely, all tensors of this form satisfy $D^{(n)} = 0$.
\end{corollary}

Such tensors represent the joint probability distributions of $n$ binary random variables $X_1, \ldots, X_n$ that satisfy the independence statement $ X_i \perp \{ X_1, X_j \} $, for all $i \neq j $. The full independence model is the special case $d_1 = \cdots = d_n = 0$.

\begin{proof}
The hypothesis that $D^{(n)} = 0$ means all terms in the sum-of-squares certificate for $D^{(n)}$ vanish. We assume that the first and second determinants, $d_1$ and $d_2$, are non-zero. Without loss of generality, it suffices to show that the third determinant vanishes.

Write out the second flattening of the tensor, arranging the columns in two blocks according to the value of the first index
$$ T^{(2)} = \begin{bmatrix} \leftarrow & a_{00*} & \rightarrow & \leftarrow & a_{10*} & \rightarrow \\ \leftarrow & a_{01*} & \rightarrow & \leftarrow & a_{11*} & \rightarrow \end{bmatrix} .$$
All $2 \times 2$ minors upon which the first index is constant appear as terms in the sum-of-squares certificate for $D^{(n)}$ (see the proof of Theorem \ref{conv}). Therefore the left and right hand halves of $T^{(2)}$ are two rank one matrices. Say they are given by multiples of vectors ${\bf x}$ and ${\bf y}$ respectively, of length $2^{n-2}$. We write
$$ T^{(2)} = \begin{bmatrix} t_0 {\bf x} & s_0 {\bf y} \\ t_1 {\bf x} & s_1 {\bf y} \end{bmatrix} .$$

We now write the third flattening in terms of vectors ${\bf x}$ and ${\bf y}$. We write ${\bf x} = \begin{bmatrix} {\bf x}_0 & {\bf x}_1 \end{bmatrix}$, where the entries of ${\bf x}$ are arranged according to the value of the third index: ${\bf x}_0$ are those entries of the tensor with a 0 in their third index, and ${\bf x}_1$ are those with a 1 in their third index. Similarly for ${\bf y}$. We can then write the third flattening as
$$ T^{(3)} = \begin{bmatrix} t_0 {\bf x}_0 & t_1 {\bf x}_0 & s_0 {\bf y}_0 & s_1 {\bf y}_0 \\ t_0 {\bf x}_1 & t_1 {\bf x}_1 & s_0 {\bf y}_1 & s_1 {\bf y}_1 \end{bmatrix} .$$
Just as for the second flattening, we have organized the columns of the third flattening according to the value of the first index. So the matrix is formed of two rank one matrices concatenated side-by-side. This implies that there exists vectors ${\bf x}'$ and ${\bf y}'$ such that
\begin{equation} \label{t3} T^{(3)} = \begin{bmatrix} \alpha_0 t_0 {\bf x}' & \alpha_0 t_1 {\bf x}' & \beta_0 s_0 {\bf y}' & \beta_0 s_1 {\bf y}' \\ \alpha_1 t_0 {\bf x}' & \alpha_1 t_1 {\bf x}' & \beta_1 s_0 {\bf x}' & \beta_1 s_1 {\bf y}' \end{bmatrix} .\end{equation}
The term
$$ {( a_{01{\bf j}}a_{10{\bf i}} + a_{00{\bf i}}a_{11{\bf j}} - a_{01{\bf i}}a_{10{\bf j}} - a_{00{\bf j}}a_{11{\bf i}})}^2 $$
appears in a sum-of-squares certificate for $D^{(n)}$, for all ${\bf i}$ and ${\bf j}$, as follows. Let $m$ be such that ${\bf i}$ and ${\bf j}$ differ in $m-2$ indices. Projecting to the $m$ indices consisting of these and the first two, we obtain one of the combinations of six minors from $D^{(m)}_m$ depicted in Figure \ref{sg}. Hence it must be zero. Substituting in our expression in \eqref{t3} for the entries of the tensor yields the equation
$$ ( \alpha_1 \beta_2 s_1t_2 + \alpha_2 \beta_1 t_1 s_2 - \alpha_2 \beta_1 s_1 t_2 - \alpha_1 \beta_2 s_2 t_1 ) x_{\bf i}' y_{\bf j}' = 0, \quad \text{for all ${\bf i}$ and ${\bf j}$}$$
where the entry of ${\bf x}'$ corresponding to multi-index ${\bf i}$ is denoted $x_{{\bf i}}'$, and likewise for ${\bf y}'$. Hence one of ${\bf x}'$ and ${\bf y}'$ must be zero, which contradicts $T^{(2)}$ being full rank, or
$( \alpha_2 \beta_1 - \alpha_1 \beta_2 ) ( s_2 t_1 - s_1 t_2 ) = 0$
which shows that $T^{(3)}$ is rank one, and hence $d_3 = 0$, as required.
\end{proof}

The following example shows that the above inequalities in the Gram determinants do not always hold for tensors of size $m_1 \times m_2 \times \cdots \times m_n$ with some $m_i > 2$.

\begin{example}
Consider the $2 \times 2 \times 3$ tensor with entries
$$ T_{111} = \frac{1}{\sqrt{2}}, \quad T_{213} = \frac{1}{\sqrt{2}}, \quad T_{ijk} = 0 \text{ otherwise } .$$
A computation shows that $d_1 = \frac{1}{4}$ while $d_2 = d_3 = 0$. This tensor can be appropriately included into larger tensor formats to show the result for fixed larger sizes.
\end{example}

\section{The Semi-Algebraic Description}\label{3}

We seek a semi-algebraic description for the Gram locus, 
the image of $\mathcal{G}(\mathcal{B})$. We begin with the case $n = 3$, where
$$\mathcal{G} : \R^2 \otimes \R^2 \otimes \R^2 \to \R^3, \quad (a_{ijk}) \mapsto (d_1, d_2, d_3) ,$$ 
and
$$ \mathcal{B} = \left\{ (a_{ijk}) \in \R^2 \otimes \R^2 \otimes \R^2 : \sum_{ijk} a_{ijk}^2 \leq 1 \right\} .$$

\begin{proof}[Proof of Theorem \ref{thm:neq3}]
We first find the Zariski closure of the boundary of the image $\mathcal{G}(\mathcal{B})$. Following the approach in \cite{KPS}, this is contained in the branch locus of the map $\mathcal{G}$ and that of its restriction to the boundary $\partial \mathcal{B} =  \left\{ (a_{ijk}) \in \R^2 \otimes \R^2 \otimes \R^2 : \sum_{ijk} a_{ijk}^2 = 1 \right\} $. These branch loci are $p$ and $q$ respectively, obtained by direct computation (using the boxed code on page 14):
$$ p = d_1 d_2 d_3 (d_1 - d_2 )(d_1 - d_3) (d_2 - d_3),$$
$$ \hspace{11ex} q = \prod_{i < j} (d_i - d_j ) \times \prod_{i} \left( d_i - \frac{1}{4} \right) \times Q \text{ , } \qquad \text{where}$$
\vspace{-2ex}
\begin{multline*} Q = \prod_{i = 1}^3 \left( \sum_{j \neq i} d_j - d_i \right) - \frac{1}{2} \times \prod_{ (i,j,k) \in \{ \pm 1 \}^3 } ( i \sqrt{d_1} + j \sqrt{d_2} + k \sqrt{d_3} ) = Q_1 - Q_2 \\ \hspace{3.7ex} =  (d_1 + d_2 - d_3 )( d_1 - d_2 + d_3 ) ( - d_1 + d_2 + d_3 ) - \frac{1}{2} {\left( d_1^2 + d_2^2 + d_3^2 - 2( d_1 d_2 + d_1 d_3 + d_2 d_3 ) \right)}^2 .\end{multline*}
The polynomial $Q$ is the non-linear part of the boundary, depicted in Figure \ref{fig1}. The Zariski closure of the boundary of $\mathcal{G}(\mathcal{B})$ is contained in the vanishing locus of $p$ and $q$, $V(pq)$.

\vspace{-2ex}

\begin{figure}[H]
\centering
\includegraphics[width = 8.6cm]{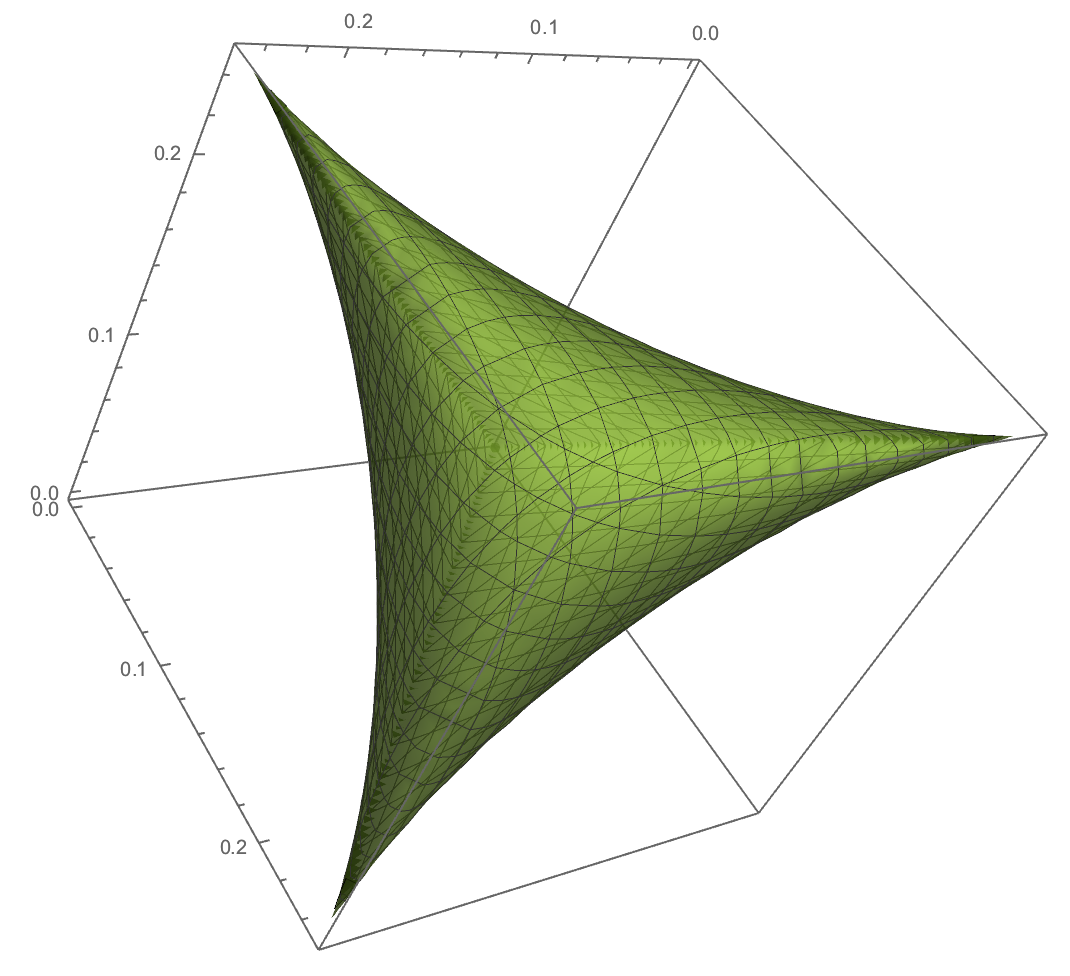}
\caption{The surface $Q = 0$}
\label{fig1}
\end{figure}

The image $\mathcal{G}(\mathcal{B})$ is the closure of the union of some connected components in $\R^3 \backslash V(pq)$: each connected component is either contained in the image, or disjoint from it. Hence it suffices to consider components contained inside the convex hull of $\mathcal{G}(\mathcal{B})$. Figure \ref{fig1} shows that $[0,\frac{1}{4}]^3 \backslash V(Q)$ has five connected components. The connected component containing $( \frac{1}{4} -\epsilon, \epsilon, \epsilon)$, for $\epsilon > 0$ sufficiently small, intersects the set $d_1 > d_2 + d_3$, hence by Theorem \ref{conv} it is not contained in the image. There are three such components by symmetry. The interior of the surface $V(Q)$ is contained in the convex hull of the image. Likewise for the component containing the point $(\frac{1}{4} - \epsilon, \frac{1}{4} - \epsilon, \frac{1}{4} - \epsilon)$, for $\epsilon > 0$ sufficiently small. A direct computation finds tensors that map to each connected component of $[0,\frac{1}{4}]^3 \backslash V(pq)$ in these two last pieces, hence they are the image $\mathcal{G}(\mathcal{B})$.

It remains to find the semi-algebraic description. The interior of the surface $V(Q)$ is given by $Q \geq 0$. The surface $Q = Q_1 - Q_2$ meets the plane $d_1 = \frac{1}{4}$ along the planar curve ${(d_2 - d_3)}^2 +\frac{1}{2} ( d_2+ d_3) - \frac{3}{16} $ with multiplicity two. Imposing that all three such polynomials, obtained by relabeling, be positive yields the component of $[0, \frac{1}{4}]^3 \backslash V(Q)$ containing the point $(\frac{1}{4} - \epsilon, \frac{1}{4} - \epsilon, \frac{1}{4} - \epsilon)$.\end{proof}

\begin{figure}[h]
\centering
\includegraphics[width = 5cm]{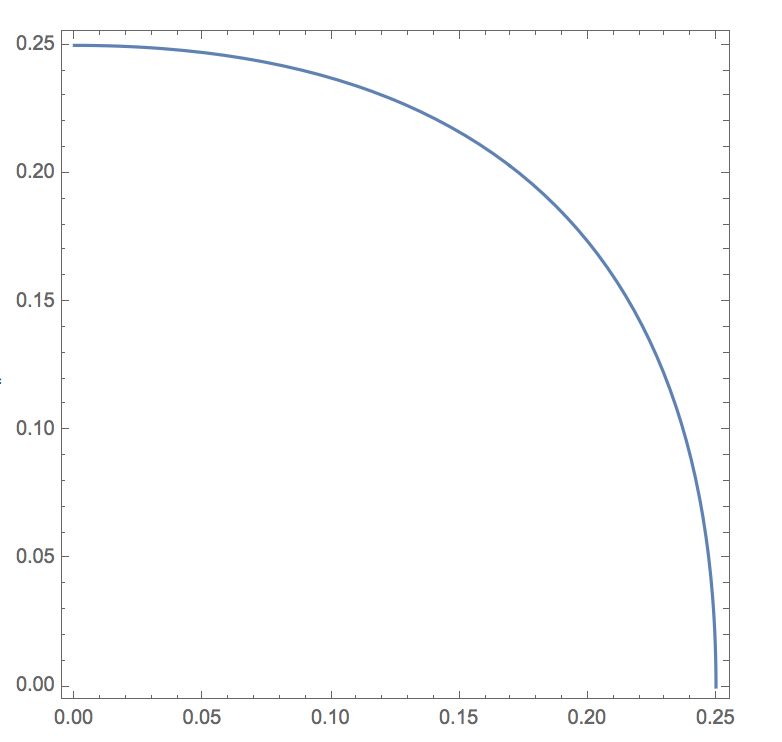}
\caption{The surface $Q = 0$ meets the plane $d_1 = 1/4$}
\label{fig2}
\end{figure}


Polynomials $p$ and $q$ from the proof of Theorem \ref{thm:neq3} are computed in Macaulay2 as follows. Computational speed-ups are obtained by changing coordinates from the $a_{ijk}$, the eight entries of the array, to coordinates $x_{ijk}$ that are invariant under the orthogonal group $O_2 \times O_2 \times O_2$. The variables \verb|di| refer to the determinants, while \verb|t| is the trace of any flattening.

\vspace{2ex}
\cprotect\fbox{
  \begin{minipage}{0.8\textwidth}
Make two ideals (using the $x_{ijk}$ coordinates):
\begin{verbatim}
C1 = minors(3,jacobian(ideal(d1,d2,d3)));
C2 = minors(4,jacobian(ideal(d1,d2,d3,t)))+ideal(1-t);
\end{verbatim}
Saturate with respect to the known ramification locus:
\begin{verbatim}
c = ideal((d1 - d2)*(d1 - d3)*(d2 - d3));
C1 = C1:c; C2 = C2:c;
\end{verbatim}
Project $C_1$ and $C_2$ to the ring $\Q[d_1,d_2,d_3]$ to obtain $p$ and $q$ respectively. The computation takes 5 minutes.
  \end{minipage}
}
\vspace{1ex}

Section \ref{1} shows how to convert determinantal constraints to the higher order singular value coordinates. In \cite{HU}, the authors work in the three-dimensional space of the highest singular values from each flattening. The image of $Q=0$ in these coordinates is depicted in Figure \ref{fig3}. The point of the star near $(1,1,1)$ is the true algebraic description for the experiments with random tensors in \cite[Figure 3.1]{HU}.

\begin{figure}[H]
\centering
\includegraphics[width = 9cm]{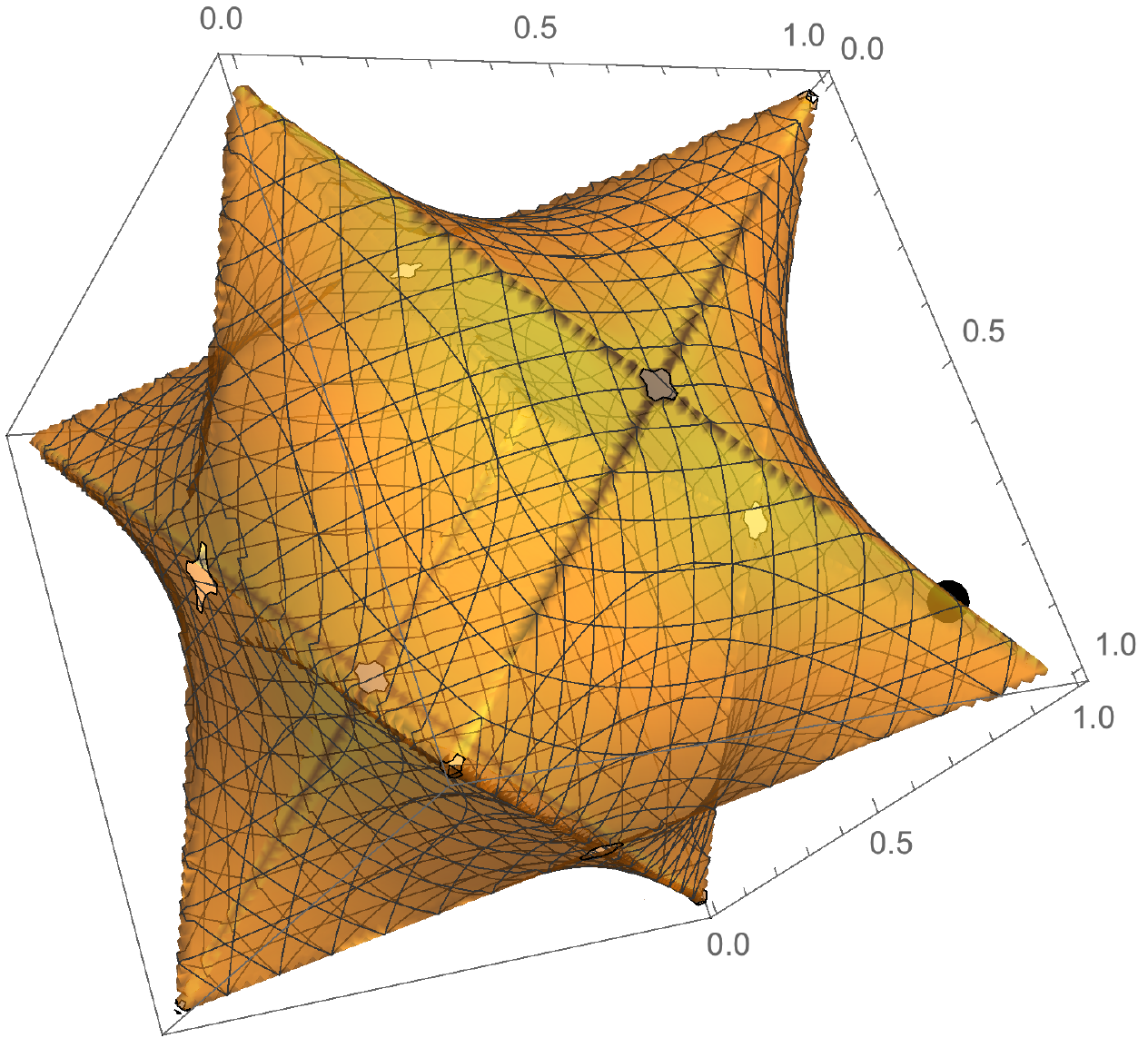}
\caption{The surface $Q = 0$ in singular value coordinates. The black dot is Example \ref{counter}.}
\label{fig3}
\end{figure}

Conjecture \ref{conj:main} describes, for $n \geq 4$, the image $\mathcal{G}(\mathcal{B})$ inside the cube $[0, \frac{1}{4}]^n$ by the single polynomial inequality $Q_1 \geq Q_2$. The reason for this discrepancy with the $n=3$ case can be understood by evaluating $Q_1 - Q_2$ when all $d_i = \frac{1}{4}$. We obtain
$$ Q_1 \left( \frac{1}{4}, \ldots, \frac{1}{4} \right) = {\left(\frac{n - 2}{4}\right)}^n , \quad Q_2 \left( \frac{1}{4}, \ldots, \frac{1}{4} \right) = \frac{1}{2^{2^n+1}} \prod_{k = 0}^n {( n - 2k)}^{{ n \choose k}} .$$
The value of $Q_2$ is $0$ for all even $n$. Among odd $n$, the difference $Q_1 - Q_2$ grows in $n$, and is positive for all $n \geq 5$. Hence the connected component of the complement of $V(Q_1 - Q_2)$ containing the point $(\frac{1}{4} - \epsilon, \frac{1}{4} - \epsilon, \frac{1}{4} - \epsilon)$, for some small $\epsilon > 0$, is the same as the piece $Q_1 - Q_2 \geq 0$ for all $n \geq 4$. The sufficiency of Conjecture \ref{conj:main} has been tested for one million tensors, with $4 \leq n \leq 7$.





The boundary of the Gram locus $\mathcal{G}(\mathcal{B})$ contains parts of all the hyperplanes $d_i = \frac{1}{4}$. If a tensor of norm one lies on the hyperplane $d_i = \frac{1}{4}$, its singular values in the $i$th flattening are both $\frac{1}{\sqrt{2}}$ and, in particular, are the same. However, the following example shows that not all tensors on the boundary of the Gram locus have two singular values the same in some flattening. Since the change of coordinates given by Equation \eqref{coord} does not map boundary points to the interior, this disproves the conjecture stated in Section 1 of \cite{HU}.

\begin{example}\label{counter}
Consider the tensor
$$ T_{111} = \frac{1}{\sqrt[4]{2}}, \qquad T_{101} = T_{011} = \sqrt{\frac{1}{2} - \frac{1}{ 2 \sqrt{2} }} , \quad T_{ijk} = 0 , \text{ otherwise }.$$
Its tuple of Gram determinants,
$$(d_1, d_2, d_3) = \left(\frac{1}{8}, \frac{1}{8}, \frac{ \sqrt{2} - 1}{2} \right),$$
lies on the part of $V(Q)$ that contributes to the boundary of $\mathcal{G}(\mathcal{B})$. The higher order singular values are: the square roots of $\frac{1 + \sqrt{2}}{2 \sqrt{2}}$ and $\frac{1}{2} - \frac{1}{2 \sqrt{2}}$, for flattenings one and two, and the square roots of $\frac{1}{\sqrt{2}}$ and $1 - \frac{1}{\sqrt{2}}$ in the third flattening. It is labeled in Figure \ref{fig3} by a black dot which can be seen to lie on the boundary hypersurface.
\end{example}

\section{The Fibers}\label{4}

We conclude the paper with a discussion of the fibers of the Gram determinant map
$$ \mathcal{G}^{-1} (d_1 , \ldots, d_n ) \subseteq \R^{2} \otimes \cdots \otimes \R^{2}.$$
Each fiber is defined by $n$ non-homogeneous quartics in the space of binary tensors. It consists of a union of orbits under the orthogonal equivalence action for binary tensors, $O_2 \times \cdots \times O_2$ \cite[Proposition 2.2]{HU}. Dimension counting reveals that the quotient
$$ \mathcal{G}^{-1} (d_1 , \ldots, d_n ) / \left( O_2 \times \cdots \times O_2 \right) $$
has dimension exponentially sized in $n$. Recall from Section \ref{1} that, for tensors of fixed norm, the Gram determinant map $\mathcal{G}$ corresponds to the map sending a tensor to its higher order singular values. Hence, while in the matrix case the singular values define a matrix up to orthogonal equivalence, the same is not true of tensors and their higher order singular values. 

Distinguishing between distinct equivalence classes inside the fiber would allow this troublesome gap to be bridged. A direct computation proves the following for the case $n = 3$.

\begin{theorem}\label{2x2x2}
A $2 \times 2 \times 2$ tensor is defined up to orthogonal equivalence by its higher order singular values and its hyperdeterminant.
\end{theorem}

The hyperdeterminant is the unique (up to scale) $SL_2 \times SL_2 \times SL_2$ invariant. It is given by the formula
\vspace{-1ex}
$$ 
 a_{000}^2 a_{111}^2
+a_{001}^2 a_{110}^2
+a_{010}^2 a_{101}^2
+a_{011}^2 a_{100}^2
+4 a_{000} a_{011} a_{101} a_{110}
+4 a_{001} a_{010} a_{100} a_{111} $$
\vspace{-4.5ex}
$$
-2 a_{000} a_{001} a_{110} a_{111}
-2 a_{000} a_{010} a_{101} a_{111}
-2 a_{000} a_{011} a_{100} a_{111}$$
\vspace{-4ex}
$$
-2 a_{001} a_{010} a_{101} a_{110}
-2 a_{001} a_{011} a_{100} a_{110}
-2 a_{010} a_{011} a_{100} a_{101} 
.$$
Theorem \ref{2x2x2} says that if two tensors are related by a change of basis and have the same higher order singular values, they are related by an {\em orthogonal} change of basis. The result extends to tensors of multilinear rank $(2,2,\ldots,2)$ by projecting to the minimal subspaces. Letting
$ t = \sum_{ij\ldots k} a_{ij\ldots k}^2 $,
the fibers of the map
$$ (d_1 , d_2, d_3, t, {\rm hyperdet} ) : \R^2 \otimes \R^2 \otimes \R^2 \to \R^5 ,$$
are the equivalence classes of tensors under the $O_2 \times O_2 \times O_2$ orthogonal equivalence action. Successful extension of Theorem \ref{2x2x2} to higher $n$ would yield a summary of tensors up-to-orthogonal-equivalence. 

We lastly consider the map that sends a tensor of general format $m_1 \times \cdots \times m_n$ to its higher order singular values, the singular values of each principal flattening. Given a fixed tensor $T$, the tensors $S$ in the same fiber as $T$ are those whose $i$th principal flattening is orthogonally equivalent to the $i$th principal flattening of $T$, for all $1 \leq i \leq n$. In particular, the first flattening is orthogonally equivalent to the first flattening of $T$, hence $S \in ( O_{m_1} \times O_{m_2 \cdots m_n} ) \cdot T$. Repeating for all $1 \leq i \leq n$, we obtain that the fiber is exactly those tensors $S$ for which
$$ S \, \, \in \, \, \bigcap_i ( O_{m_i} \times O_{m_1 \cdots \widehat{m_i} \cdots m_n} \cdot T ) .$$
However, the element of the group $O_{m_i} \times O_{m_1 \cdots \widehat{m_i} \cdots m_n}$ will be different for each $i$. Indeed,
$$ \bigcap_i \left( O_{m_i} \times O_{m_1 \ldots \widehat{m_i} \ldots m_n} \right) = O_{m_1} \times \cdots \times O_{m_n} ,$$
so tensors in the same fiber, which also differ by the same matrices in each flattening, are actually orthogonally equivalent. We can only express the fiber as the above intersection of $n$ orbits, not as a single orbit. It is an open problem to extend Theorem \ref{2x2x2} to larger tensor formats: to add minimal additional invariants such that the fibers are single orthogonal equivalence classes. 

\subsection*{Acknowledgements}

I would like to thank my advisor Bernd Sturmfels for helpful discussions. Thanks also to Jan Draisma for information about an orthogonally invariant basis for $2 \times 2 \times 2$ tensors, and to Andre Uschmajew and Nick Vannieuwenhoven for useful comments. I received partial funding from the Pachter Lab and NIH grant R01HG008164.




\begin{small}

\end{small}
\end{document}